\documentclass[12pt]{article}
\usepackage{color}

\usepackage{amsmath}
\usepackage{amsfonts}
\usepackage{amsthm}
\usepackage{amssymb}
\usepackage{bm}
\usepackage[applemac]{inputenc}
\usepackage{enumerate}
\makeatletter

\newtheoremstyle{mystyle}{}{}{\slshape}{2pt}{\scshape}{.}{ }{} 

\newtheorem{thm}{Theorem}[section]
\newtheorem{cor}[thm]{Corollary}

\newtheorem{prop}[thm]{Proposition}
\newtheorem{lemme}[thm]{Lemma}
\newtheorem{fait}[thm]{Fact}

\theoremstyle{definition}
\newtheorem{defi}[thm]{Definition}
\theoremstyle{mystyle}

\theoremstyle{remark}

\newcommand{\monster}{\mathcal U}
\newcommand{\la}{\langle}
\newcommand{\ra}{\rangle}

\DeclareMathOperator{\Av}{Av}

\def\indsym#1#2{%
 \setbox0=\hbox{$\m@th#1x$}%
 \kern\wd0%
 \hbox to 0pt{\hss$\m@th#1\mid$\hbox to 0pt{$\m@th#1^{#2}$\hss}\hss}%
 \lower.9\ht0\hbox to 0pt{\hss$\m@th#1\smile$\hss}%
 \kern\wd0}

\def\nindsym#1#2{%
 \setbox0=\hbox{$\m@th#1x$}%
 \kern\wd0%
 \hbox to 0pt{\hss$\m@th#1\not$\kern1.4\wd0\hss}
 \hbox to 0pt{\hss$\m@th#1\mid$\hbox to 0pt{$\m@th#1^{#2}$\hss}\hss}%
 \lower.9\ht0\hbox to 0pt{\hss$\m@th#1\smile$\hss}%
 \kern\wd0}

\title{VC-sets and generic compact domination}
\author{Pierre Simon\footnote{Partially supported by the European Research Council under the European Unions Seventh Framework Programme (FP7/2007-2013) / ERC Grant agreement no. 291111. Partially supported by ValCoMo (ANR-13-BS01-0006)}\\ CNRS, Universit\'e Lyon 1}
\date{}

\begin{document}
\maketitle

\begin{abstract}
Let $X$ be a closed subset of a locally compact second countable group $G$ whose family of translates has finite VC-dimension. We show that the topological border of $X$ has Haar measure 0. Under an extra technical hypothesis, this also holds if $X$ is constructible. We deduce from this generic compact domination for definably amenable NIP groups.
\end{abstract}

Let $G$ be a locally compact second countable group and fix a Haar measure $\mu$. Let $\mathcal F\subseteq \mathfrak P(G)$ be a family of subsets of $G$. We say that $\mathcal F$ has finite VC-dimension if for some $n$, for every subset $A\subset G$ of size $n$, we have $\mathcal F \cap A \neq \mathfrak P(A)$, where $\mathcal F \cap A = \{S\cap A : S\in \mathcal F\}$. This notion has been well studied (in fact for families of subsets of arbitrary measure spaces). A number of combinatorial and measure theoretic properties are known: most notably families of finite VC-dimension have $\epsilon$-approximations (Fact \ref{fact_epsilonnet}). See \cite[Chapter 10]{Matousek} for more information. In the first section of this paper, we assume that the family $\mathcal F$ is stable under translation and our aim is to show, under further technical conditions, that the sets in $\mathcal F$ cannot be too fractal-like. For this, we might as well assume that $\mathcal F$ is exactly the family of translates of some set $X\subseteq G$. Our main theorem is the following.

\begin{thm}
Let $X\subseteq G$ and assume that the family of left-translates of $X$ has finite VC-dimension. Assume furthermore that either $X$ is closed or that the following holds for both $Y=X$ and $Y=G\setminus X$:

\quad $Y$ is $F_\sigma$ and for all $x\in Y$ and neighborhood $U$ of $x$, $\mu(U\cap Y)>0$.

\noindent
\underline{Then} the topological border $\partial X$ of $X$ has Haar measure 0.
\end{thm}

This theorem has a precise motivation coming from model theory: In the second part of the paper, we use it to prove the \emph{generic compact domination conjecture} for definably amenable NIP groups. To explain this, we now assume that the reader is familiar with model theory and in particular NIP theories.

Let $T$ be an NIP theory and $G$ a definable group in $T$. As usual, we denote by $G^{00}$ the smallest type-definable subgroup of $G$ of bounded index (which exists by NIP). We say that $G$ is definably amenable if it admits a $G$-invariant finitely additive probability measure on definable subsets of $G$. This notion was first studied in \cite{NIP1}, \cite{NIP2} and more systematically explored in \cite{tamedyn}. The papers \cite{NIP1} and \cite{NIP2} study in particular \emph{fsg} groups which are definably amenable groups which enjoy stronger stable-like properties. The typical example of an fsg group is a definably compact group in an o-minimal theory. Such a group has a unique invariant measure $\mu$. Let now $S\subset S_G(\monster)$ be the set of global generic types of $G$, equivalently (under the fsg assumption) the set of types $p\in S_G(\monster)$ which are $G^{00}$-invariant. Let $X$ be a definable subset of $G$ and let $h$ be the Haar measure on the compact group $G/G^{00}$. The generic compact domination conjecture from \cite{NIP2} says that for $h$-almost every coset $gG^{00}$ of $G^{00}$, the set $S\cap gG^{00}$ lies either entirely inside $X$ or entirely outside $X$.

A proof of this conjecture was announced in \cite{NIP3}, but there is a mistake there: the proof of Theorem 4.3 is flawed. We will give here the first (hopefully) correct proof of the conjecture. In fact, we prove a more general statement that holds for all definably amenable groups, making use of the work done in \cite{tamedyn}.

We also deduce as a corollary the following result: If $G$ is definably amenable NIP and $\mu$ is a global measure which is $G^{00}$-invariant and $G(M_0)$-invariant for a small model $M_0$, then $\mu$ is actually $G$-invariant.

\smallskip
Thanks to the referee for a number of helpful remarks.

\section{Preliminaries}

Let $G$ be a topological space. A constructible set $X\subseteq G$ is a finite boolean combination of closed sets. If $X\subseteq G$, we let $\partial X$ denote the border of $X$: $\partial X=\overline X\setminus \mathring X = \overline X \cap \overline{X^c}$. We also let $X^{ext}$ denote the exterior of $X$, that is the interior of $X^c$.

\smallskip
We now recall some facts about VC-dimension without proofs. A good reference on the subject is \cite[Chapter 10]{Matousek}.

Let $X$ be any set and $\mathcal F\subseteq \mathfrak P(X)$ a family of subsets of $X$. For $X_0\subseteq X$, we write $\mathcal F \cap X_0 = \{S\cap X_0 : S\in \mathcal F\}$ and we say that $X_0$ is \emph{shattered} by $\mathcal F$ if $\mathcal F \cap X_0 = \mathfrak P(X_0)$.

We say that the family $\mathcal F$ has VC-dimension $n$ if it shatters some set of size $n$, but no set of size $n+1$. We say that it has infinite VC-dimension if it does not have finite VC-dimension, that is if for every $n$, some subset $X_n \subseteq X$ of size $n$ is shattered by $\mathcal F$. The exact value of the VC-dimension will not be important to us here. What is important is the dichotomy between finite and infinite VC-dimension.

One can also define the dual VC-dimension of $\mathcal F$ as the largest integer $n$ (or $\infty$ if there is none) such that there are $X_1,\ldots,X_n \in \mathcal F$ for which all the $2^n$ cells in the Venn diagram they generate are non-empty. An easy but important observation is that $\mathcal F$ has finite VC-dimension if and only if it has finite dual VC-dimension (though the two may not coincide). We will use this freely in this text.

The only theorem we will need about families of finite VC-dimension is the following fundamental fact (first proved in \cite{VC}). In the statement, and later in the text, $\Av(x_1,\ldots,x_n; S)$ stands for $\frac 1 n |\{i\leq n : x_i \in S\}|$.

\begin{fait}[$\epsilon$-approximations] \label{fact_epsilonnet}
For any $k > 0$ and $\varepsilon > 0$ there is $N \in \mathbb N$ satisfying the following.

If $(X, \mu)$ is a probability space, and $\mathcal{F}$ a family of subsets of $X$ of VC-dimension $\leq k$ such that:
\begin{enumerate}
\item every set in $\mathcal{F}$ is measurable;
\item for each $n$, the function $f_n: X^n \to [0,1]$ given by $$(x_1, \ldots, x_n) \mapsto \sup_{S \in \mathcal{F}} | \Av(x_1, \ldots, x_n; S) - \mu(S) |$$  is measurable;
\item for each $n$, the function $g_n: X^{2n} \to [0,1]$
$$(x_1, \ldots, x_n, y_1, \ldots, y_n) \mapsto \sup_{S \in \mathcal{F}} | \Av(x_1, \ldots, x_n; S) - \Av(y_1, \ldots, y_n; S) |$$
is measurable.
\end{enumerate}
\underline{Then} there are some $x_1, \ldots, x_N \in X$ (possibly with repetitions) such that for any $S \in \mathcal{F}$,
$$\left | \Av(x_1,\ldots,x_N;S) - \mu(S) \right | < \epsilon.$$\end{fait}

We check now that the measurability conditions are satisfied when working with the family of translates of a single measurable set in a locally compact group.

\begin{prop}\label{prop_measurability}
Let $G$ be a second countable locally compact group equipped with a Haar measure $\mu$. Let $U,X\subseteq G$ be Borel sets. Then the family $\mathcal F = \{gX : g\in U\}$ of $U$-translates of $X$ satisfies the assumptions 1, 2, 3 in Fact \ref{fact_epsilonnet} with respect to $(G,\mu)$.
\end{prop}
\begin{proof}
Assumption 1 is clear since Borel sets are measurable.

Define $f_n$ and $g_n$ as in assumptions 2 and 3 and we need to show that $f_{n}$ and $g_{n}$ are measurable.

Since all the elements of $\mathcal F$ have the same measure $\mu(X)$, we have

$$f_n(x_1,\ldots,x_n)=\sup_{g\in U} |\Av(x_1,\ldots,x_n;g X) -\mu(X)|.$$

Note that $\Av(x_1,\ldots,x_n;g X)$ can take only finitely many values. It is then enough to show that for a fixed $I\subseteq n$, the set 
$$A_I=\{(x_1,\ldots,x_{n})\in G^n \mid \text{ for some }g\in U,~ x_i\in g X \Leftrightarrow i\in I\}$$
 is measurable. But we can write $A_I$ as the projection of $A'_I\subseteq G^{n+1}$ where $A'_I$ is the intersection of $\{(g,x_1,\ldots,x_{n}) : g^{-1}x_i \in X, g\in U\}$ for $i\in I$ and $\{(g,x_1,\ldots,x_{n}) : g^{-1}x_i \notin X, g\in U\}$ for $i\notin I$. As group multiplication is continuous and both $X$ and $U$ are Borel, those sets are Borel as well. Hence $A_I$ is analytic. Now $G$ is a Polish space (\cite[9.A]{Kechris1995}) and analytic subsets of Polish spaces are universally measurable (that is measurable for any $\sigma$-finite Borel measure; see e.g. \cite[Theorem 29(7)]{Kechris1995}). In particular they are measurable with respect to the Haar measure $\mu$.
 
Finally, measurabilty of the function $g_n$ follows at once from measurability of the sets
$$B_I=\{(x_1,\ldots,x_{n},x_{n+1},\ldots,x_{2n})\in G^{2n} \mid \text{ for some }g\in U,~ x_i\in g X \Leftrightarrow i\in I\},$$

where now $I\subseteq 2n$.
\end{proof}

\section{VC-sets in locally compact groups}

In this section, we fix a second countable locally compact group $G$ and a left-invariant Haar measure $\mu$ on it.

For any integer $k$ and set $A$, we let ${}^kA$ denote the set of functions from $k=\{0,\ldots,k-1\}$ to $A$.

\begin{defi}
A set $X\subseteq G$ is called a \emph{VC-set} if the family $\{g X : g\in G\}$ of left translates of $X$ has finite VC-dimension.
\end{defi}



\begin{defi}
A pair $(V_0,V_1)$ of disjoint open subsets of $G$ is called a \emph{tame pair} if there is a maximal $n$ for which we can find $g_0,\ldots,g_{n-1}$ and $(x_\eta:\eta\in {}^n\{0,1\})$ in $G$ such that for all $k<n$ and $\eta$, $g_k x_\eta \in V_{\eta(k)}$.
\end{defi}

In particular, if $X$ is a VC-set, then $(\mathring X, X^{ext})$ is a tame pair (using VC-duality). This terminology is inspired from similar notions arising in the study of tame dynamical systems (see e.g. \cite{huang}).

We recall the following well-known theorem of Steinhaus which holds in any locally compact group (see e.g. \cite{steinhaus} for a short proof).

\begin{fait}[Steinhaus]
Let $X\subseteq G$ be a measurable set of positive Haar measure. Then the set $XX^{-1}=\{xy^{-1}:x,y\in X\}$ contains a neighborhood of the identity.
\end{fait}
%

We now prove the main technical theorem of this paper.

\begin{thm}\label{th_bordernil}
Let $(V_0,V_1)$ be a tame pair in $G$, then $\mu(\overline{V_0} \cap \overline{V_1})=0$.
\end{thm}
\begin{proof}
Set $\delta = \overline{V_0}\cap \overline{V_1}$ and assume that $\mu(\delta)>0$. Let $\Delta=\{x\in \delta : \mu(U\cap \delta)>0$ for all neighborhoods $U$ of $x\}$.

Let $(O_i:i\in \mathbb N)$ be a countable base of open sets of $G$. For any $x\in \delta \setminus \Delta$, there is some $i(x)\in \mathbb N$ such that $x\in O_{i(x)}$ and $\mu(O_{i(x)}\cap \delta)=0$. Then $\delta\setminus \Delta=\bigcup_{x\in \delta\setminus \Delta} O_{i(x)}\cap \delta$ is a countable union of measure 0 sets and hence has measure 0. In particular $\mu(\Delta)>0$ and for $x\in \Delta$ and $U$ a neighborhood of $x$, $\mu(U\cap \Delta)>0$.

Fix some integer $n$ and we will construct inductively points $(g_i:i<n)$ in $G$ and $(x_{\eta}:\eta\in {}^k\{0,1\})$ in $\Delta$, $k\leq n$, such that 
$$\forall l~\forall \eta\in {}^l\{0,1\}~ \forall k<l \quad g_k x_\eta \in V_{\eta(k)}.$$
This will contradict the fact that $(V_0,V_1)$ is tame.

\smallskip\noindent
\underline{Step 0}:
By Steinhaus' theorem, let $U$ be a neighborhood of 1 such that $U\subseteq \Delta \Delta^{-1}$. Let $x_{\la 0\ra}\in \Delta$ be any point. Take some $g_0^0 \in U$ such that $g_0^0x_{\la 0\ra}\in V_0$. By definition of $U$, there is $x_{\la 1\ra}\in \Delta$ such that $g_0^0x_{\la 1\ra}\in \Delta$. Then there is some $g_0^1\in G$ such that $g_0^1g_0^0 x_{\la 1\ra}\in V_1$ and $g_0^1g_0^0x_{\la 0\ra}\in V_0$ (because $V_0$ is open and $g_0^0x_{\la 1\ra}\in \overline{V_1}$). Set $g_0=g_0^1g_0^0$ and we finish step 0.

\smallskip\noindent
\underline{Step $l$}: Assume that we have defined $(x_\eta:\eta\in {}^l\{0,1\})$ and $(g_i:i<l)$. As $V_0,V_1$ are open, we can find open neighborhoods $W_\eta$ of $x_\eta$ such that for any $x' \in W_\eta$ and $k<l$, we have $g_k x' \in V_{\eta(k)}$. Set $X_\eta=W_\eta \cap \Delta$. By construction of $\Delta$, $X_\eta$ has positive measure. Apply Steinhaus' theorem to find some open neighborhood~$U_l$ of 1 such that $U_l \subseteq X_\eta X_\eta^{-1}$ for each $\eta$.

Now enumerate ${}^{l+1}\{0,1\}$ in an arbitrary order as $(\eta_0,\ldots,\eta_{2^{l+1}-1})$. We define elements $(g_l^i:i<2^{l+1})$ and $(x_{\eta_i}:i<2^{l+1})$ inductively such that for all $j\leq i<2^{l+1}$:

$\bullet_0$ $g_l^i\cdots g_l^0\in U_l$;

$\bullet_1$ $g_l^i\cdots g_l^0x_{\eta_j} \in V_{\eta_j(l)}$;

$\bullet_2$ for $\lambda \in {}^l\{0,1\}$ and $\alpha \in \{0,1\}$, $x_{\lambda\hat{~}\alpha}\in X_\lambda$.

\noindent
Once this is done, set $g_l=g_l^{2^{l+1}-1}\cdots g_l^0$ to finish step $l$: condition $\bullet_1$ ensures that multiplication by $g_l$ sends the $x_\eta$'s in the right $V_i$ and $\bullet_2$ ensures this for multiplication by all previous $g_k$'s.

Assume that we have achieved this for all $j<i$. Set $\eta=\eta_i$, $\alpha=\eta(l)$ and $\lambda\in {}^l\{0,1\}$ such that $\eta=\lambda \hat{~}\alpha$. Let $x_\eta\in X_\lambda$ such that $g_l^{i-1}\cdots g_l^0x_\eta\in \Delta$ (possible by $\bullet_0$ and construction of $U_l$). Then find some $g_l^i \in G$ small enough such that:

$g_l^i\cdots g_l^0\in U_l$;

$g_l^i\cdots g_l^0x_{\eta_j}\in V_{\eta_j(l)}$ for each $j<i$;

$g_l^i\cdots g_l^0x_{\eta}\in V_\alpha$.

\noindent
This is possible as $U_l$, $V_0$ and $V_1$ are open and $g_l^{i-1}\cdots g_l^0x_\eta \in \overline{V_\alpha}$. This finishes the construction.
\end{proof}

\begin{lemme}\label{lem_interior}
Let $X\subseteq G$ be a VC-set which is $F_\sigma$. If $\mu(X)>0$, then $X$ has non-empty interior.
\end{lemme}
\begin{proof}
Fix some compact symmetric neighborhood $U$ of the identity. Take a compact set $B'$ such that $\mu(X\cap B')=:\epsilon >0$. Let $B=UB'$. Then $B$ is compact and hence has finite measure. Note that for each $g\in U$, we have $gX\cap B \supseteq g(X\cap B')$, hence $\mu(gX \cap B)\geq \mu(X\cap B')=\epsilon$.

Let $\mathcal F = \{g X \cap B:g\in U\}$ be the family of $U$-translates of $X$ intersected by $B$ so as to remain in a space of finite measure. It has finite VC-dimension as $X$ is a VC-set. Since all the elements of $\mathcal F$ have measure $>\epsilon$, by Fact \ref{fact_epsilonnet} (and Proposition \ref{prop_measurability}) applied to the Haar measure restricted to $B$, there is a finite set of points $\{x_0,\ldots,x_{n-1}\}$ meeting every translate $g X$, $g\in U$. Then for each $g\in U$, there is $k<n$ such that $g^{-1}x_k\in X$. In other words, $U\subseteq \bigcup_{i<n} X\cdot x_i^{-1}$. As each of the sets $X\cdot x_i^{-1}$ is $F_\sigma$, by the Baire property one of them must have non-empty interior and then so does $X$.
\end{proof}

\begin{thm}\label{th_closed}
Let $F\subseteq G$ be a closed VC-set. Then $\mu(\partial F)=0$.
\end{thm}
\begin{proof}
Assume for a contradiction that $\mu(\partial F)>0$ and define $\Delta F=\{x\in \partial F: \mu(U\cap \partial F)>0$ for all neighborhoods $U$ of $x\}$. Then as in the proof of Theorem \ref{th_bordernil}, $\mu(\partial F \setminus \Delta F)=0$ and $\mu(\Delta F)>0$. Let $x\in \Delta F$ and let $U'$ be a compact neighborhood of $x$. Then $r:=\mu(F\cap U')\geq \mu(\Delta F\cap U')>0$. Take also $W$ a symmetric compact neighborhood of 1 and let $U=WU'$, so that for each $g\in W$, $\mu(gF\cap U)\geq r$. The family $\{gF\cap U: g\in W\}$ has finite VC-dimension and is included in the finite measure space $U$. We conclude as in Lemma \ref{lem_interior}: by the VC-theorem, there is a finite set $\{x_0,\ldots,x_{n-1}\}\subseteq U$ meeting every $g F$, $g\in W$. Then for some $i$, $Fx_i^{-1} \cap W$ has non-empty interior. Hence $F\cap WU$ has non-empty interior. As we can take $U$ and $W$ arbitrarily small, $WU$ is an arbitrarily small neighborhood of $x$, and hence $x\in \overline{\mathring F}$.

The assumption on $F$ implies that $(\mathring F,F^{c})$ is a tame pair. By the previous theorem, $\delta:=\overline{\mathring F} \cap \overline{F^{c}}$ has measure 0. But we have shown that $\Delta F\subseteq \delta$, hence also $\Delta F$ has measure 0.
\end{proof}

In the statement above, one cannot replace \emph{closed} by \emph{constructible} (finite boolean combination of closed sets). Here is a counterexample in $(\mathbb R,+)$: take $K\subseteq \mathbb R$ a Cantor set of positive Lebesgue measure. Write $\mathbb R\setminus K$ as a countable union of disjoint intervals $((a_i,b_i):i\in \mathbb N)$. For each $i\in \mathbb N$ consider an increasing sequence $(c_k^i:k\in \mathbb Z)$ in $(a_i,b_i)$ whose limits at $\pm \infty$ are respectively $a_i$ and $b_i$. Let $X$ be the union of all those sequences. Then $X$ is constructible, indeed locally closed, because it is discrete. Also $\partial X=\overline X=K\cup X$. In particular $\mu(\partial X)>0$. To make $X$ into a VC-set simply choose the points $c_k^i$ so that the map $X\times X\to \mathbb R$, $(x,y)\mapsto y-x$ is injective. Then it is easy to see that the family of translates of $X$ cannot shatter a set of size 3.

However, if $X$ is a finite boolean combination of closed VC-sets, then the theorem does hold for $X$, because the border of a boolean combination is included in the union of the borders of the sets in question.

We now state a second version of the theorem which applies in particular for constructible sets, under an extra assumption.

\begin{thm}\label{th_constr}
Let $X$ be a VC-set and assume that the following condition holds both for $Y=X$ and $Y=X^c$:

\quad $Y$ is $F_\sigma$ and for all $x\in Y$ and neighborhood $U$ of $x$, $\mu(U\cap Y)>0$.

\noindent
Then $\mu(\partial X)=0$.
\end{thm}
\begin{proof}
As $X$ is a VC-set, the pair $(\mathring X, X^{ext})$ is tame. By Theorem \ref{th_bordernil}, it is enough to show that $\partial X = \overline{\mathring X} \cap \overline{X^{ext}}$. Let $x\in \partial X$. By symmetry of the roles of $X$ and $X^c$ it is enough to show that $x\in \overline{\mathring X}$. Assume first that $x\in X$. Take $U$ a neighborhood of $x$ of finite measure. Then $X\cap U$ has positive measure by assumption. By the same reasoning as in Theorem \ref{th_closed}, we conclude that $X\cap U$ has non-empty interior, which gives what we want. Now assume that $x\in X^{c}$. Let $U$ be a neighborhood of $x$. As $x\in \partial X$, there is some $x'\in U\cap X$. Then by the argument in the first case, $x'$ is in $\overline{\mathring X}$. As $U$ was arbitrary, also $x$ is in $\overline{\mathring X}$.
\end{proof}

\section{Generic compact domination}

In this section, we apply the previous results to prove a conjecture in model theory concerning definably amenable NIP groups. We assume familiarity with model theory and in particular NIP theories.

Let $T$ be an NIP theory and $M_0\models T$ a small model. Let also $G$ be an $M_0$-definable group. We assume that $G$ is definably amenable, which means that there exists a $G$-invariant Keisler measure on $G$. See \cite[Chapter 8]{NIPBook}, \cite{tamedyn} for background on this notion.

Recall that $G$ admits a smallest type-definable group of bounded index $G^{00}$ and that the quotient $K=G/G^{00}$ equipped with the logic topology is a compact Hausdorff group. We let $h$ denote the normalized Haar measure on it.

Recall the notion of f-generic type (in the sense of \cite{tamedyn}): a global type $p$ concentrating on $G$ is f-generic if for any $\phi(x)\in p$, there is a small model $M$ such that no (left-)translate of $\phi(x)$ forks over $M$. Equivalently, $p$ is $G^{00}$-invariant (by left translations). Such a type gives rise to an invariant measure $\mu_p$ on $G$ defined by $\mu_p(\phi)= h(\{\bar g\in G/G^{00} : \bar g\cdot p \vdash \phi(x)\})$ for any definable subset $\phi(x)$ of $G$. If $p$ is weakly-random for $\mu_p$---which means that $\mu_p(X)>0$ for any $X\in p$---then we call $p$ an almost periodic type. This is equivalent to asking that $\overline{G \cdot p}$, the closure of the orbit of $p$ in $S(\monster)$, is a minimal $G$-invariant closed subset of $S(\monster)$. All this is explained in \cite{tamedyn}.

Let $p$ be an almost periodic type concentrating on $G^{00}$, and let $\pi: \overline{G\cdot p} \to G/G^{00}$ be the canonical projection. For $\phi(x)$ a definable set of $G$---which we identify with the corresponding clopen subset of $S(\monster)$---let $X=d_p(\phi)\subseteq G/G^{00}$ be the set of $\bar g\in G/G^{00}$ such that $\bar g\cdot p\vdash \phi(x)$. Finally, define
$$E_\phi=\{\bar g\in G/G^{00} : \pi^{-1}(\bar g)\cap \phi\neq \emptyset\text{ and }\pi^{-1}(\bar g)\cap \neg\phi \neq \emptyset\}.$$

The following is shown in \cite{tamedyn}.

\begin{fait}\label{fact_bairedom}
We have the inclusion $E_\phi \subseteq \partial X$ and the set $E_\phi$ is a closed meager set.
\end{fait}

We can now state the compact domination theorem.

\begin{thm}\label{th_compactdom}
The set $E_\phi$ has Haar measure 0.
\end{thm}
\begin{proof}
First we show that we may restrict to the case where $L$ is countable. Let $L_0$ be a countable sublanguage containing $\phi$ and the formulas defining $G$ and its group structure. Let $G^{00}_{L_0}$ be $G^{00}$ in the sense of $L_0$. We have a canonical surjection $\pi_0:G/G^{00} \to G/G^{00}_{L_0}$. Let $E^0_\phi$ be the set $E_\phi$ in the sense of $L_0$. Then $E_\phi \subseteq \pi_0^{-1}(E^0_\phi)$. Also the Haar measure on $G/G^{00}_{L_0}$ is the pushforward of the Haar measure on $G/G^{00}$ (by uniqueness). It is now enough to show that $E^0_\phi$ has Haar measure 0. Therefore we may replace $L$ by $L_0$ and assume that $L$ is countable. Then $G/G^{00}$ is a second countable compact group.

\underline{Claim 1}: $X=d_p(\phi)$ is a constructible VC-set in $G/G^{00}$.

Proof: The fact that $d_p(\phi)$ is constructible was shown in \cite{tamedyn}, it easily follows from Borel-definability of invariant types in NIP theories (Proposition 2.6 in \cite{NIP2}). To see that $d_p(\phi)$ is a VC-set, let $\{x_0,\ldots,x_{n-1}\}\subseteq K$ be a finite set shattered by the family of translates of $d_p(\phi)$. For $i\subseteq n$, let $y_i$ be such that $x_k \in y_i\cdot d_p(\phi)\iff k\in i$.

For each $k<n$, $i\subseteq n$, pick representatives $g_k$ of $x_k$ and $h_i$ of $y_i$ in $G(\monster)$. Finally, let $a$ realize $p$ over all those points. We see that 
$$\models \phi(h_i^{-1}g_ka) \iff k\in i.$$

This shows that $n$ is at most the VC-dimension of the formula $\psi(xy;z)=\phi(z^{-1}\cdot x\cdot y)$ and concludes the proof of the claim.

\smallskip
\underline{Claim 2}: For $x_*\in X$ and $U\subseteq G/G^{00}$ an open neighborhood of $x_*$, $h(U\cap X)>0$.

Proof: By construction of the topology on $G/G^{00}$, there is a definable set $\psi(x)\in L(M_0)$ such that $\pi^{-1}(\{x_*\})\subseteq \psi(x)\subseteq \pi^{-1}(U)$. Take $g\in G(\monster)$ projecting to $x_*$. As $p$ concentrates on $G^{00}$, we have $\pi(g\cdot p)=x_*$ and $g\cdot p\vdash \psi(x)$. Therefore $g\cdot p\vdash \psi(x)\wedge \phi(x)$ and hence $\mu_p(\psi(x)\wedge \phi(x))>0$. This exactly means that $h(d_p(\psi)\cap d_p(\phi))>0$. As $p$ concentrates on $G^{00}$, we have $d_p(\psi)\subseteq U$ and thus $h(U\cap X)>0$.

\smallskip
Replacing $\phi(x)$ by $\neg \phi(x)$ we obtain the same result for $X^c$. Now we can apply Theorem \ref{th_constr} to obtain $\mu(\partial X)=0$. Then also $\mu(E_\phi)=0$ by Fact \ref{fact_bairedom}.
\end{proof}

\subsection{$G(M_0)$-invariant measures}

We keep notations as above: $G$ is a definably amenable NIP group defined over some model $M_0$. If $\mu(x)$ is a measure over $M$, then the support of $\mu$ is the (closed) set of types $p\in S_x(M)$ such that $p\vdash \phi(x) \Longrightarrow \mu(\phi(x))>0$.

We will need the following facts. First an easy property of measures.

\begin{fait}[\cite{tamedyn}, Proposition 3.15]\label{fact_g00}
A global measure $\mu$ on $G$ is $G^{00}$-invariant if and only if all types in its support are f-generic.
\end{fait}

Next a model-theoretic adaptation of Fact \ref{fact_epsilonnet}; here $\Av(p_1,\ldots,p_n;\phi(x))$ means $\frac 1 n |\{i : p_i \vdash \phi(x)\} |$.

\begin{fait}[\cite{NIPBook}, Proposition 7.11]\label{fact_nipappr}
Let $\mu(x)$ be any Keisler measure over a model $M$. Let $\phi(x;y)\in L$ and fix $\epsilon >0$. Then there are types $p_1,\ldots,p_n$ in the support of $\mu$ such that for any $b\in M$,
$$| \Av(p_1,\ldots,p_n;\phi(x;b)) - \mu(\phi(x;b)) | <\epsilon.$$
\end{fait}

\begin{prop}\label{prop_apprdense}
Let $p\in S_G(\monster)$ be a global f-generic type. Fix a formula $\phi(x)\in L(\monster)$ and $\epsilon>0$. Then there are $g_1,\ldots,g_n\in G(M_0)$ such that for any $g\in G(\monster)$, $$\left | \mu_p(\phi(x)) - \Av(g_1\cdot p,\ldots,g_n\cdot p ; \phi(g^{-1}x)) \right | < \epsilon.$$
\end{prop}
\begin{proof}
First we show that we may assume that $L$ is countable. Let $L_0$ be a countable sublanguage sufficient to define $G$ and $\phi(x)$. Then the reduct of $p$ to $L_0$ is also f-generic. Letting $G^{00}_{L_0}$ be $G^{00}$ in the sense of $L_0$, we have a canonical map $\pi_0:G/G^{00} \to G/G^{00}_{L_0}$ and the Haar measure on $G/G^{00}_{L_0}$ is the pushforward of the Haar measure on $G/G^{00}$. One sees then that the measure $\mu_p$ in the sense of $L_0$ agrees with $\mu_p$ in the sense of $L$ for $L_0$-formulas. Hence we may take $L$ to be countable and then $G/G^{00}$ is a compact second countable group.

Define $X = d_p(\phi(x))$ as above, namely $X = \{\bar g\in G/G^{00} : \bar g\cdot p \vdash \phi(x)\}$. So we have $\mu_p(\phi(x)) = h(X)$. Fix a left-invariant metric $d$ on $G/G^{00}$. For $r>0$, let $V_{0,r}\subseteq \mathring X$ be the set of points at distance $>r$ from $(\mathring X)^c$ and $V_{1,r}\subseteq X^{ext}$ be the set of points at distance $>r$ from $\overline{\mathring X}$. Then $\mathring X = \bigcup_{r>0} V_{0,r}$ and hence there is $r>0$ such that $\mu(V_{0,r})>\mu(\mathring X)-\epsilon$. Similarly, we can take $r$ so that $\mu(V_{1,r})>\mu(X^{ext})-\epsilon$. Fix such an $r$ and for $i=0,1$, let $V_i = V_{i,r}$.

By regularity of the Haar measure, there are closed sets $C_0\subseteq V_0$ and $C_1\subseteq V_1$ such that $h(V_i\setminus C_i)<\epsilon$ for $i=0,1$. Let $\pi : G \to G/G^{00}$ be the canonical projection. One can find definable sets $\theta_0$ and $\theta_1$ such that $\pi^{-1}(C_i)\subseteq \theta_i \subseteq \pi^{-1}(V_i)$, $i=0,1$. Let $q$ be any f-generic type concentrating on $G^{00}$ and define $X_i = d_q(\theta_i)$. Then we have $C_i \subseteq X_i \subseteq V_i$ and $X_i$ is a Borel VC-set as in the proof of Theorem \ref{th_compactdom}. By the VC-theorem \ref{fact_epsilonnet} (and Proposition \ref{prop_measurability}), we can find points $\bar h_1,\ldots,\bar h_n \in G/G^{00}$ such that
$$\left | h(\bar g X_i) - \Av(\bar h_1,\ldots,\bar h_n;\bar gX_i) \right | \leq \epsilon,$$
for all $\bar g\in G/G^{00}$ and $i=0,1$.

Now $\pi(G(M_0))$ is dense in $G/G^{00}$, hence we can find points $g_k \in G(M_0)$, $k\leq n$, such that $d(\bar g_k,\bar h_k)<r$, where $\bar g_k = \pi(g_k)$. Fix any $\bar g\in G/G^{00}$. Then if $\bar h_k \in \bar g X_0$, we have $\bar g_k \in \bar g \mathring X$ since $\bar g X_0\subseteq \bar g V_{0,r}$. Similarly, if $\bar h_k \in \bar g X_1$, then $\bar g_k \in \bar gX^{ext}$.

We then have:
$$\Av(\bar h_1,\ldots,\bar h_n;\bar g X_0)\leq \Av(\bar g_1,\ldots,\bar g_n;\bar g X)\leq 1-\Av(\bar h_1,\ldots,\bar h_n;\bar g X_1).$$
By construction of the $\bar h_k$'s, $\Av(\bar h_1,\ldots,\bar h_n;\bar g X_i)\geq h(\bar g X_i) - \epsilon \geq h(\bar g V_i) - 2\epsilon$. By Theorem \ref{th_compactdom}, $h(\mathring X)=h(X)$, therefore $h(\bar g V_0)\geq h(\bar g X)- \epsilon$. Also $h(X^{ext})=1-h(X)$, hence $1-h(\bar gV_1)\leq h(\bar g X)+\epsilon$. Putting it all together, we see that $$h(\bar g X)-3\epsilon\leq \Av(\bar g_1,\ldots,\bar g_n;\bar g X) \leq h(\bar g X)+ 3\epsilon,$$
from which the result follows.
\end{proof}

\begin{cor}
Let $p\in S_G(\monster)$ be almost periodic (equivalently, $p$ is f-generic and weakly random for $\mu_p$), then $\overline{G(M_0)\cdot p} = \overline{G(\monster)\cdot p}$.
\end{cor}
\begin{proof}
Assume that $\overline{G(M_0)\cdot p} \subsetneq \overline{G(\monster)\cdot p}$, then there is a formula $\phi(x)$ which is disjoint from the first set, but not from the second. But then $\mu_p(\phi(x))>0$ by almost periodicity which contradicts Proposition \ref{prop_apprdense}.
\end{proof}

\begin{cor}
The minimal flows in the dynamical system $(G(M_0), S_G(M_0))$ are precisely the projections of the minimal flows in $(G(\monster),S_G(\monster))$.
\end{cor}
\begin{proof}
Let $\pi:S_G(\monster) \to S_G(M_0)$ be the natural projection. Let first $Y\subseteq S_G(\monster)$ be a minimal flow. So $Y$ has the form $\overline{G\cdot p}$ for some almost periodic type $p$. By the previous corollary, $G(M_0)\cdot p$ is dense in $Y$, hence also $G(M_0)\cdot \pi(p)$ is dense in $\pi(Y)$. As this is true for all $p\in Y$, $\pi(Y)$ is a minimal flow.

The converse is true without any assumptions on $T$ or $G$: let $X\subset S_G(M_0)$ be a minimal flow under the action of $G(M_0)$. Let $p\in X$ be a type and let $\tilde p$ be an heir of $p$ over $\monster$. Given $\phi(x)\in L(M_0)$ disjoint from $X$, we have that $p\vdash \neg \phi(g\cdot x)$ for any $g\in G(M_0)$ as $X$ is a flow. By the heir property, also $\tilde p\vdash \neg \phi(g\cdot x)$ for any $g\in G(\monster)$. This means that the orbit $G(\monster)\cdot p$ lies entirely inside $\pi^{-1}(X)$, hence so does its closure $\overline{G(\monster)\cdot p}$, which is a subflow. Let $Y$ be a minimal $G(\monster)$-subflow inside that closure. Then $\pi(Y)\subseteq X$ is a $G(M_0)$-subflow, but then $\pi(Y)=X$ by minimality of $X$.
\end{proof}

The point of the following theorem is that---at least when $L$ is countable---the study of invariant measures on $G$, can be reduced to the study of invariant measures for the action of a countable group on a compact space (which is the situation for many theorems in topological dynamics). Here, the compact space would be the space of $G^{00}(\monster)$-invariant types and the countable group would be $G(M_0)$, for $M_0$ a countable model.

\begin{thm}
Let $G$ be a definably amenable NIP group and let $\mu$ be a global $G^{00}(\monster)$-invariant and $G(M_0)$-invariant measure. Then $\mu$ is $G(\monster)$-invariant.
\end{thm}
\begin{proof}
By Fact \ref{fact_g00}, the support of $\mu$ consists only of f-generic types. Fix a formula $\phi(x)\in L(\monster)$ and $\epsilon>0$. By Fact \ref{fact_nipappr}, there are f-generic types $p_1,\ldots,p_n$ such that for any $g\in G(\monster)$, we have $$|\mu(\phi(gx)) - \Av(p_1,\ldots,p_n;\phi(gx))| \leq \epsilon.$$
Proposition \ref{prop_apprdense} is stated for one f-generic type, but works just as well for finitely many types: instead of doing the proof for just one $X=d_p(\phi)$, we have finitely many $X^i = d_{p_i}(\phi)$; we do the same analysis for each of them separately and apply the VC-theorem for all of them at once. This gives us points $g_1,\ldots,g_m\in G(M_0)$ such that for each $i\leq n$ and for each $g\in G(\monster)$,
$$|\mu_{p_i}(\phi(gx)) - \Av_j(g_j\cdot p_i;\phi(gx))|\leq \epsilon.$$
By $G(M_0)$-invariance of $\mu$, we also have that for any $g$ and $j\leq m$,
$$|\mu(\phi(gx)) - \Av(g_j\cdot p_1,\ldots,g_j\cdot p_n;\phi(gx))| \leq \epsilon.$$
Hence averaging over all $p_i$'s and $g_j$'s we obtain
$$\left | \frac 1 n \sum_{i\leq n} \mu_{p_i}(\phi(gx)) - \mu(\phi(gx)) \right | \leq 2 \epsilon.$$
and in particular
$$\left | \frac 1 n \sum_{i\leq n} \mu_{p_i}(\phi(x)) - \mu(\phi(x)) \right | \leq 2 \epsilon.$$

As the measure on the left-hand side is $G$-invariant, we deduce $|\mu(\phi(gx))-\mu(\phi(x))|\leq 4 \epsilon$, hence $\mu$ is $G$-invariant.
\end{proof}

\subsection{fsg groups}

An NIP group $G$ is \emph{fsg} if there is some type $p\in S_G(\monster)$ and a small model $M_0$ such that all translates of $p$ are finitely satisfiable in $M_0$. It is shown in \cite{NIP2} that such a group admits a unique invariant measure $\mu$, in particular it is definably amenable. Also, any set of positive measure is generic (finitely many translates cover the group). Let $p$ be a generic type, equivalently a type weakly-random for $\mu$, then $\overline {G\cdot p}$ is exactly the set of all generic types. Also, f-generic formulas and generic formulas coincide (\cite{HPS_note} or \cite{tamedyn}). We let $\pi: \overline{G\cdot p} \to G/G^{00}$ be the canonical projection. Theorem \ref{th_compactdom} gives us the following.

\begin{cor}
Let $G$ be an fsg group and let $S\subset S_G(\monster)$ be the set of generic types of $G$. Let $\phi$ be a definable subset of $G$ and define
$$E_\phi=\{\bar g\in G/G^{00} : \pi^{-1}(\bar g)\cap \phi\neq \emptyset\text{ and }\pi^{-1}(\bar g)\cap \phi^c \neq \emptyset\}.$$

Then $E_\phi$ has Haar measure 0.
\end{cor}

\begin{cor}
Let $G$ be fsg, then there is a unique Keisler measure on $G$ which is $G^{00}$-invariant and lifts the Haar measure on $G/G^{00}$.
\end{cor}
\begin{proof}
The fact that the previous corollary implies this one is already proved in \cite[Proposition 5.7]{NIP3}, but we repeat the argument.

Let $\lambda$ be a global Keisler measure on $G$ which is $G^{00}(\monster)$-invariant and lifts the Haar measure on $G/G^{00}$. By Fact \ref{fact_g00}, the support of $\lambda$ is composed of f-generic---and thus generic---types, so $\lambda$ can be seen as a Borel measure on the set $S$ of global generic types. Compact domination then implies that it is entirely determined by its image $\pi_*(\lambda)$ on $G/G^{00}$: Let $\phi(x)$ be a definable set. Then as $E_\phi$ has measure 0, we must have $\lambda(\phi)= \pi_*\lambda(\pi(\phi)) = h(\pi(\phi))$. This proves uniqueness of $\lambda$.
\end{proof}

\bibliography{tout}
\bibliographystyle{alpha}

\end{document}